\documentclass[a4paper, 12pt]{amsart}

\newtheorem{theorem}{Theorem}[section]
\newtheorem{lemma}[theorem]{Lemma}
\newtheorem{proposition}[theorem]{Proposition}

\theoremstyle{definition}
\newtheorem{definition}[theorem]{Definition}
\newtheorem{remark}[theorem]{Remark}
\newtheorem{example}[theorem]{Example}

\begin{document}
\date{2017-8-22}
\title[Relative weak mixing of W*-dynamical systems]{Relative weak mixing of W*-dynamical systems via joinings}
\author{Rocco Duvenhage and Malcolm King}
\address{Department of Physics\\
University of Pretoria\\
Pretoria 0002\\
South Africa}

\email{rocco.duvenhage@up.ac.za}

\address{Department of Mathematics and Applied Mathematics\\
University of Pretoria\\
Pretoria 0002\\
South Africa}

\begin{abstract}
A characterization of relative weak mixing in W*-dynamical systems in terms of
a relatively independent joining is proven.

\end{abstract}
\maketitle

\section{Introduction}

This paper studies relative weak mixing for W*-dynamical systems in terms of
joinings. Here a W*-dynamical system refers to a von Neumann algebra with a
faithful normal tracial state which is invariant under the dynamics, given by
iteration of a fixed $\ast$-automorphism of the von Neumann algebra (i.e. we
focus exclusively on actions of the group $\mathbb{Z}$). The main result is a
characterization of relative weak mixing in terms of relative ergodicity of
the relative product of the system with its mirror image on the commutant (in
the cyclic representation). The relative product system is defined using the
relatively independent joining obtained from the conditional expectation onto 
the von Neumann subalgebra relative to which we are working. Generalizing the 
classical case, the subalgebra in question is always taken to be globally 
invariant under the dynamics of the W*-dynamical system.

The proof involves a careful analysis of the interplay between the von Neumann
algebra, its commutant, and the conditional expectation. Some results of 
independent interest obtained on the way to the main result, do not require the 
state to be tracial. In this case, we need to restrict ourselves to subalgebras 
which are globally invariant under the modular group, to ensure the existence 
of the conditional expectation.

In classical ergodic theory it is well known that a dynamical system is weakly
mixing if and only if its product with itself is ergodic. Our main result is
essentially noncommutative and relative version of this.

A noncommutative theory of joinings has been developed in \cite{D1}, \cite{D2}
and \cite{D3}, generalizing some aspects of the classical theory (see \cite{G}
for a thorough treatment, and \cite{F67} as well as \cite{Rud} for the
origins). It included a study of weak mixing, relative ergodicity and compact
subsystems. Subsequent work was done in \cite{BCM}, which among other things
developed various characterizations of joinings and also obtained a more
complete theory for weak mixing, building on an approach to noncommutative
joinings outlined in \cite[Section 5]{KLP}. Earlier work related to
noncommutative joinings appeared in \cite{ST}, connected to entropy, and
\cite{F1}, regarding ergodic theorems.

An investigation of relative weak mixing is a natural next step in the
development of the theory of noncommutative joinings. Relative weak mixing has
already been studied and used very effectively in the noncommutative context
in \cite{P} and \cite{AET}, but not from a joining point of view.

In particular, the authors of \cite{AET} proved quite a remarkable structure 
theorem, namely that an asymptotically abelian W*-dynamical system is weakly 
mixing relative to the center of the von Neumann algebra. This allowed them to
apply classical ergodic results to the system on the center, and then extend 
these results to the noncommutative system. They defined relative weak mixing 
in terms of a certain ergodic limit, which is the approach taken in this paper 
as well. However, we adapt their definition to a form which is more convenient 
in the proof of our main result. The two definitions are nevertheless 
equivalent when the invariant state is tracial. To prove this, we make use of 
the semi-finite trace obtained in the basic construction from the von Neumann 
algebra and the subalgebra relative to which we are working.

Since systems which are not asymptotically abelian do occur, we do not assume
asymptotic abelianness in this paper. 

Furthermore, systems can be weakly mixing relative to nontrivial subalgebras
other than the center. This includes cases where the von Neumann algebra of
the system is a factor (i.e. when the center is trivial). Therefore we work
relative to more general von Neumann subalgebras.

In the classical case, relative weak mixing is often defined in terms of a
relatively independent joining, or relative product, illustrating the
importance of this characterization in the classical case. However, it is in
many cases just stated for ergodic systems, since any system can be decomposed
into ergodic parts. See for example \cite[Theorem 7.5]{F77}, \cite[Definition
7.9]{Z2} and \cite[Definition 9.22]{G}. But we note that in \cite{FK} and
\cite[Definition 6.2]{F81}, on the other hand, ergodicity is not assumed.

In the noncommutative case the assumption of ergodicity becomes problematic,
as typically some form of asymptotic abelianness is required to do an ergodic
decomposition. See for example \cite[Subsection 4.3.1]{BR1} for an exposition.
Therefore we study the joining characterization of relative weak mixing
without the assumption of ergodicity. In particular the proof of our main
result has to deal with the difficulty of the system not being ergodic.

A number of other noncommutative relative ergodic properties have already been
studied in the literature, for example in \cite{DM}, building on ideas from
\cite{F2}, which was based in turn on variations of unique ergodicity as
studied in \cite{AD}. Those properties, however, are more of a topological
nature, rather than purely measure theoretic in origin, if one thinks in terms
of classical ergodic theory, and the techniques involved are quite different
from those in this paper.

The required background on relatively independent joinings is reviewed in
Section \ref{afdROB}, which also sets out much of the notation used later in
the paper. The definition of relative weak mixing is formulated in Section
\ref{afdRSV}. Some relevant characterizations in terms of ergodic limits are
then derived. A noncommutative example is subsequently presented to illustrate
the points made above regarding asymptotic abelianness, the center, and
ergodicity. The main result of the paper, and its proof, appear in Section
\ref{afdBewys}.

\section{Relatively independent joinings\label{afdROB}}

For convenience we summarize the special case of relatively independent
joinings that we need here, along with some additional definitions.
Simultaneously this fixes notation that will be used throughout the paper. We
use the same setup as in \cite{D1, D2, D3}. Please refer in particular to
\cite[Sections 2 and 3]{D3} for further discussion.

In the remainder of this paper W*-dynamical systems are referred to simply as
``systems'' and they are defined as follows:

\begin{definition}
\label{stelsel}A \emph{system} $\mathbf{A}=\left(  A,\mu,\alpha\right)  $
consists of a faithful normal state $\mu$ on a (necessarily $\sigma$-finite)
von Neumann algebra $A$, and a $\ast$-automorphism $\alpha$ of $A$, such that
$\mu\circ\alpha=\mu$.
\end{definition}

In the rest of the paper, the symbols $\mathbf{A}$, $\mathbf{B}$ and
$\mathbf{F}$ denote systems $\left(  A,\mu,\alpha\right)  $, $\left(
B,\nu,\beta\right)  $ and $\left(  F,\lambda,\varphi\right)  $. For
$\mathbf{A}$ we assume without loss that $A$ is a von Neumann algebra on the
Hilbert space $H$, with $\mu$ given by a cyclic and separating vector
$\Omega\in H$, i.e.%
\[
\mu(a)=\left\langle \Omega,a\Omega\right\rangle
\]
for all $a\in A$.

\begin{definition}
A \textit{joining} of $\mathbf{A}$ and $\mathbf{B}$ is a state $\omega$ on the
algebraic tensor product $A\odot B$ such that $\omega\left(  a\otimes
1_{B}\right)  =\mu(a)$, $\omega\left(  1_{A}\otimes b\right)  =\nu(b)$ and
$\omega\circ\left(  \alpha\odot\beta\right)  =\omega$ for all $a\in A$ and
$b\in B$.
\end{definition}

The modular conjugation associated to the state $\mu$, will be denoted by $J$,
and we let
\[
j:B(H)\rightarrow B(H):a\mapsto Ja^{\ast}J.
\]

The dynamics $\alpha$ of a system $\mathbf{A}$ can be represented by a unitary
operator $U$ on $H_{\mu}$ defined by extending
\[
Ua\Omega:=\alpha(a)\Omega.
\]
It satisfies
\[
UaU^{\ast}=\alpha(a)
\]
for all $a\in A$.

\begin{definition}
We call $\mathbf{F}$ a \emph{subsystem} of $\mathbf{A}$ if $F$ is a von
Neumann subalgebra of $A$ (containing the unit of $A$) such that $\mu
|_{F}=\lambda$ and $\alpha|_{F}=\varphi$. If $F$ is globally invariant under
modular group associated to $\mu$, then $\mathbf{F}$ is called a \emph{modular
subsystem} of $\mathbf{A}$.
\end{definition}

Throughout the rest of the paper, $\mathbf{F}$ will be a modular subsystem of
$\mathbf{A}$. Note that if the state $\mu$ of the system $\mathbf{A}$ is a
trace (i.e. $\mu(ab)=\mu(ba)$ for all $a,b\in A$), then all of its subsystems
are modular. Much of our work, in particular our main result, Theorem
\ref{st173}, is for the case where $\mu$ is tracial.

Given a system $\mathbf{A}$, carry the state and dynamics of $\mathbf{A}$ over
to $A^{\prime}$ in a natural way using $j$, by defining a state $\mu^{\prime}$
and $\ast$-automorphism $\alpha^{\prime}$ on $A^{\prime}$ by
\[
\mu^{\prime}(b):=\mu\circ j(b)=\left\langle \Omega,b\Omega\right\rangle
\]
and
\[
\alpha^{\prime}(b):=j\circ\alpha\circ j(b)=UbU^{\ast}%
\]
for all $b\in A^{\prime}$, since $UJ=JU$. This defines the system
\[
\mathbf{A}^{\prime}:=(A^{\prime},\mu^{\prime},\alpha^{\prime}).
\]

Since $\mathbf{F}$ is a modular subsystem of $\mathbf{A}$, we obtain a modular
subsystem $\mathbf{\tilde{F}}=\left(  \tilde{F},\tilde{\lambda},\tilde
{\varphi}\right)  $ of $\mathbf{A}^{\prime}$ as follows: Set
\[
\tilde{F}:=j(F)\subset A^{\prime}
\]
(note that by the symbol $\subset$ we mean inclusion, with equality allowed),
and let
\[
\tilde{\lambda}:=\mu^{\prime}|_{\tilde{F}}%
\]
and
\[
\tilde{\varphi}:=\alpha^{\prime}|_{\tilde{F}}.
\]

We can now construct the relatively independent joining of $\mathbf{A}$ and
$\mathbf{A}^{\prime}$ over $\mathbf{F}$:

Since $\mathbf{F}$ is a modular subsystem of $\mathbf{A}$, we know by
Tomita-Takesaki theory (see for example \cite[Theorem IX.4.2]{T}) that we have
a unique conditional expectation
\[
D:A\rightarrow F
\]
such that $\lambda\circ D=\mu$. Then
\[
\tilde{D}:=j\circ D\circ j:A^{\prime}\rightarrow\tilde{F}%
\]
is the unique conditional expectation such that $\tilde{\lambda}\circ\tilde
{D}=\tilde{\nu}$.

Let $P$ be the projection of $H$ onto
\[
H_{F}:=\overline{F\Omega}=\overline{\tilde{F}\Omega},
\]
where the last equality follows from $JH_{F}=H_{F}$. Then
\[
D(a)\Omega=Pa\Omega
\]
for all $a\in A$. This follows from the general construction of such
conditional expectations; see for example \cite[Section 10.2]{S}. Similarly,%
\[
\tilde{D}(b)\Omega=Pb\Omega
\]
for all $b\in A^{\prime}$. Also note that
\[
D\circ\alpha=\alpha\circ D=\varphi\circ D,
\]
and analogously for $\tilde{D}$, since
\[
PU=UP,
\]
as is easily verified from $\alpha(F)=F$.

Define the unital $\ast$-homomorphism
\[
\delta:F\odot\tilde{F}\rightarrow B(H),
\]
to be the linear extension of $F\times\tilde{F}\rightarrow B(H):(a,b)\mapsto
ab$. Defining the \emph{diagonal} state
\[
\Delta_{\lambda}:F\odot\tilde{F}\rightarrow\mathbb{C}%
\]
of $\lambda$ by
\[
\Delta_{\lambda}(c):=\left\langle \Omega,\delta(c)\Omega\right\rangle
\]
for all $c\in F\odot\tilde{F}$, allows us to define a state $\mu\odot
_{\lambda}\mu^{\prime}$ on $A\odot A^{\prime}$ by
\begin{equation}
\mu\odot_{\lambda}\mu^{\prime}:=\Delta_{\lambda}\circ E \label{mumu}%
\end{equation}
where
\[
E:=D\odot\tilde{D}.
\]
Note that $\mu\odot_{\lambda}\mu^{\prime}$ is indeed a joining of $\mathbf{A}$
and $\mathbf{A}^{\prime}$, with the property that $(\mu\odot_{\lambda}%
\mu^{\prime})|_{F\odot\tilde{F}}=\Delta_{\lambda}$, and it is called the
\emph{relatively independent joining of }$\mathbf{A}$\emph{\ and }%
$\mathbf{A}^{\prime}$\emph{\ over }$\mathbf{F}$. We also denote this joining
by
\[
\omega:=\mu\odot_{\lambda}\mu^{\prime}.
\]

We also write
\[
\mathbf{A}\odot_{\mathbf{F}}\mathbf{A}^{\prime}:=(A\odot A^{\prime},\mu
\odot_{\lambda}\mu^{\prime},\alpha\odot\alpha^{\prime})
\]
and call $\mathbf{A}\odot_{\mathbf{F}}\mathbf{A}^{\prime}$ the \emph{relative
product system} (of $\mathbf{A}$ and $\mathbf{A}^{\prime}$ over $\mathbf{F}$),
and it is an example of a $\ast$\emph{-dynamical system}, namely it consists
of a state $\omega=\mu\odot_{\lambda}\mu^{\prime}$ on a unital $\ast$-algebra
$A\odot\tilde{B}$, and a $\ast$-automorphism $\alpha\odot\alpha^{\prime}$ of
$A\odot A^{\prime}$ such that $\omega\circ(\alpha\odot\alpha^{\prime})=\omega
$. However, this is typically not a (W*-dynamical) system as given by
Definition \ref{stelsel}.

The cyclic representation of $A\odot A^{\prime}$ obtained from $\omega$ by the
GNS construction will be denoted by $(H_{\omega},\pi_{\omega},\Omega_{\omega
})$. Since $\omega$ can be extended to a state on the maximal C*-algebraic
tensor product $A\otimes_{m}A^{\prime}$ (see for example \cite[Proposition
4.1]{D2}), we know that $\pi_{\omega}$ is a $\ast$-homomorphism from $A\odot
A^{\prime}$ into the bounded operators $B(H_{\omega})$. Let%
\[
\gamma_{\omega}:A\odot A^{\prime}\rightarrow H_{\omega}:t\mapsto\pi_{\omega
}(t)\Omega_{\omega}.
\]
Furthermore, let $W$ denote the unitary representation of
\[
\tau:=\alpha\odot\alpha^{\prime}%
\]
on $H_{\omega}$, i.e. it is defined as the extension of
\[
W\gamma_{\omega}(t):=\gamma_{\omega}(\tau(t))
\]
for all $t\in A\odot A^{\prime}$.

The cyclic representation obtained from $\omega$, allows us to construct
cyclic representations $(H_{\mu},\pi_{\mu},\Omega_{\omega})$ and
$(H_{\mu^{\prime}},\pi_{\mu^{\prime}},\Omega_{\omega})$ of $(A,\mu)$ and
$(A^{\prime},\mu^{\prime})$ respectively, which are naturally embedded into
$H_{\omega}$ (as in \cite[Construction 2.3]{D1}), by setting
\[
H_{\mu}:=\overline{\gamma_{\omega}(A\otimes1)}\text{ \ \ and \ \ }\pi_{\mu
}(a):=\pi_{\omega}(a\otimes1)|_{H_{\mu}}%
\]
for every $a\in A$, and similarly for $H_{\mu^{\prime}}$ and $\pi_{\mu
^{\prime}}$.

Now we consider cyclic representations of $(F,\lambda)$ and 
$(\tilde{F},\tilde{\lambda})$:

Note that $(H_{F},\delta,\Omega_{\omega})$ is a cyclic representation of
$(F\odot\tilde{F},\Delta_{\lambda})$, since $H_{F}=\overline{\delta
(F\odot\tilde{F})\Omega}$. However, $(\overline{\gamma_{\omega}(F\odot
\tilde{F})},\pi_{\omega}|_{F\odot\tilde{F}},\Omega_{\omega})$ is also a cyclic
representation of $(F\odot\tilde{F},\Delta_{\lambda})$, so these two
representations are unitarily equivalent via the unitary operator
$V:H_{F}\rightarrow\overline{\gamma_{\omega}(F\odot\tilde{F})}$ defined as the
extension of $\delta(t)\Omega\mapsto\gamma_{\omega}(t)$ for $t\in F\odot
\tilde{F}$. Therefore%
\[
H_{\lambda}:=\overline{\gamma_{\omega}(F\otimes1)}=V\overline{\delta
(F\otimes1)\Omega}=VH_{F}=V\overline{\delta(1\otimes\tilde{F})\Omega
}=\overline{\gamma_{\omega}(1\otimes\tilde{F})},
\]
which means that $(F,\lambda)$ and $(\tilde{F},\tilde{\lambda})$ are cyclicly
represented on the same subspace $H_{\lambda}$ of $H_{\omega}$ by
\[
\pi_{\lambda}(f):=\pi_{\mu}(f)|_{H_{\lambda}}\text{ \ \ and \ \ }\pi
_{\tilde{\lambda}}(\tilde{f}):=\pi_{\mu^{\prime}}(\tilde{f})|_{H_{\lambda}}%
\]
for all $f\in F$ and $\tilde{f}\in\tilde{F}$.

\section{Relative weak mixing\label{afdRSV}}

This section presents the definition and two closely related characterizations
of relative weak mixing in terms of ergodic averages. These characterizations
do not yet involve the relative independent joining. An example of relative
weak mixing is also given.

In terms of the notation in the previous section, our main definition is the following:

\begin{definition}
\label{RSVdef}We call a system $\mathbf{A}$ \emph{weakly mixing relative to}
the modular subsystem $\mathbf{F}$ if
\begin{equation}
\lim_{N\rightarrow\infty}\frac{1}{N}\sum_{n=1}^{N}\lambda\left(
|{D(b\alpha^{n}(a))|}^{2}\right)  =0 \label{RSV}%
\end{equation}
for all $a,b\in A$ with $D(a)=D(b)=0$.
\end{definition}

In the classical case this is often also expressed by saying that $\mathbf{A}
$ is a \emph{weakly mixing extension} of $\mathbf{F}$.

\begin{remark}
We recover the absolute case of weak mixing from this definition, by using
$F=\mathbb{C}1_{A}$. Indeed, in this case we have $D(a)=\mu(a)1_{A}$ for all
$a\in A.$ Thus, Eq. (\ref{RSV}) becomes
\[
\lim_{N\rightarrow\infty}\sum_{n=1}^{N}|\mu(ba^{n}(a))|^{2}=0,
\]
or equivalently,
\begin{equation}
\lim_{N\rightarrow\infty}\frac{1}{N}\sum_{n=1}^{N}|\mu(ba^{n}(a))|=0,
\label{eq_abs_weak_mix}%
\end{equation}
for all $a,b\in A$ such that $\mu(a)=\mu(b)=0$.

The reason for this equivalence is that for any bounded sequence $(c_{n})$ of
non-negative real numbers, bounded by $c>0$, say, we have
\[
\frac{1}{N}\sum_{n=1}^{N}c_{n}^{2}\leq\frac{c}{N}\sum_{n=1}^{N}c_{n}%
\]
and, using the Cauchy-Schwarz inequality,
\begin{equation}
\frac{1}{N}\sum_{n=1}^{N}c_{n}\leq\left(  \frac{1}{N}\sum_{n=1}^{N}c_{n}%
^{2}\right)  ^{\frac{1}{2}}. \label{kwad}%
\end{equation}
Therefore,
\[
\lim_{N\rightarrow\infty}\frac{1}{N}\sum_{n=1}^{N}c_{n}^{2}=0\Leftrightarrow
\lim_{N\rightarrow\infty}\frac{1}{N}\sum_{n=1}^{N}c_{n}=0.
\]
Condition (\ref{eq_abs_weak_mix}) in turn is easily seen to be
equivalent to the following:
\[
\lim_{N\rightarrow\infty}\frac{1}{N}\sum_{n=1}^{N}|\mu(ba^{n}(a))-\mu
(b)\mu(a)|=0
\]
for all $a,b\in A$ (simply replace $a$ and $b$ by $a-\mu(a)$ and $b-\mu(b)$
respectively in Eq. (\ref{eq_abs_weak_mix}). This is the standard definition
of weak mixing.
\end{remark}

Our first simple characterization of relative weak mixing, which will also be
used in the proof of this paper's main theorem in the next section, is the following:

\begin{proposition}
\label{prop140}The system $\mathbf{A}$ is weakly mixing relative to
$\mathbf{F}$ if and only if
\begin{equation}
\lim_{N\rightarrow\infty}\frac{1}{N}\sum_{n=1}^{N}\lambda\left(  |D\left(
b\alpha^{n}(a)\right)  -D(b)D(\alpha^{n}(a))|^{2}\right)  =0 \label{RSVkar1}%
\end{equation}
for all $a,b\in A$.
\end{proposition}

\begin{proof}
Assume that $\mathbf{A}$ is weakly mixing relative to $\mathbf{F}$. For any
$a,b\in A$, setting $a_{0}:=a-D(a)$ and $b_{0}:=b-D(b)$, we have
$D(a_{0})=D(b_{0})=0$ and
\[
D(b_{0}\alpha^{n}(a_{0}))=D\left(  b\alpha^{n}(a)\right)  -D(b)D(\alpha
^{n}(a)).
\]
Hence Eq. (\ref{RSVkar1}) follows from Definition \ref{RSVdef}. The converse
is trivial by assuming either $D(a)=0$ or $D(b)=0$.
\end{proof}

This gives us variations of this characterization as well, for example,
$\mathbf{A}$ is weakly mixing relative to $\mathbf{F}$ if and only if Eq.
(\ref{RSV}) holds for all $a,b\in A$ with $D(a)=0$.

Next we are going to show that when $\mu$ is a trace, Definition \ref{RSVdef}
is equivalent to \cite[Definition 3.7]{AET}. To do this, we use the basic
construction in a similar way to how it was used in \cite[Sections 3 and
4]{AET} to prove their structure theorem.

The von Neumann algebra generated by $A$ and $P$ will be denoted by $\bar{A}$
and is referred to as the \emph{basic construction}. When $\mu$ is a trace,
then from it we obtain a faithful semifinite normal tracial weight $\bar{\mu
}:\bar{A}^{+}\rightarrow\lbrack0,\infty]$. It is also defined and tracial on
the strongly dense $\ast$-subalgebra $APA:=\operatorname{span}\{aPb:a,b\in
A\}$ of $\bar{A}$ via the equation
\[
\bar{\mu}(aPb)=\mu(ab).
\]
For more on the basic construction and the trace $\bar{\mu}$, see
\cite[Chapter 4]{SS}. Some of the early literature on this topic can be found
in \cite{Sk}, \cite{C} and \cite{J}.

We can extend the dynamics of $\alpha$ to $\bar{A}$ using the equation
\[
\bar{\alpha}(a)=UaU^{\ast}
\]
for $a\in\bar{A}$. That $\bar{\mu}\circ\bar{\alpha}=\bar{\mu}$, is
not elementary, but it is a simple consequence of the following result:

\begin{theorem}
{\cite[Theorem 4.3.11]{SS}} \label{thm_miminum_finite_lifted}Let
$\varphi$ be a weight on $\bar{A}$ with $\varphi=\bar{\mu}$ on $(APA)^{+}$. If
$\varphi$ is normal, then $\varphi=\bar{\mu}$.
\end{theorem}

Next, similar to the case of $U$, we have a unitary operator $\bar{U}:\bar
{H}\rightarrow\bar{H}$ representing $\bar{\alpha}$ on the Hilbert space
$\bar{H}$ arising from the GNS construction for $(\bar{A},\bar{\mu})$, which
is described for example in \cite[Section 7.5]{KadRing2}. We consider the
quotient space $\mathcal{N}_{\bar{\mu}}\slash  N_{\bar{\mu}}$, where
$\mathcal{N}_{\bar{\mu}}:=\{x\in\bar{A}:\bar{\mu}(x^{\ast}x)<\infty\}$ and
$N_{\bar{\mu}}:=\{x\in\bar{A}:\bar{\mu}(x^{\ast}x)=0\}$. The quotient map
$\mathcal{N}_{\bar{\mu}}\rightarrow\mathcal{N}_{\bar{\mu}}\slash  N_{\bar{\mu
}}$ sends elements $x\in\mathcal{N}_{\bar{\mu}}$ to elements
\[
\hat{x}:=x+N_{\bar{\mu}}.
\]
We may endow $\mathcal{N}_{\bar{\mu}}\slash  N_{\bar{\mu}}$ with the inner
product ${\langle}\hat{x},\hat{y}{\rangle}_{\bar{\mu}}:=\bar{\mu}(x^{\ast}y)$
for all $x,y\in\mathcal{N}_{\bar{\mu}}$. Completing $\mathcal{N}_{\bar{\mu}%
}\slash  N_{\bar{\mu}}$ in the norm ${\lVert}\hat{x}{\rVert}_{\bar{\mu}%
}:=\sqrt{{\langle\hat{x},\hat{x}\rangle}_{\bar{\mu}}}$ yields a Hilbert space
which we denote by $\bar{H}$. Since $\bar{\mu}\circ\bar{\alpha}=\bar{\mu}$,
we can define the unitary $\bar{U}:\bar{H}\rightarrow\bar{H}$ as the extension
of the map
\[
\bar{U}:\mathcal{N}_{\bar{\mu}}/N_{\bar{\mu}}\rightarrow\mathcal{N}_{\bar{\mu
}}/N_{\bar{\mu}},
\]
given by
\[
\bar{U}\hat{x}=\widehat{\alpha(x)}.
\]

In order to prove the equivalence with \cite[Definition 3.7]{AET}, we need
three lemmas which we present now. The first is just a slight variation of the
calculations that appear at the beginning of the proof of \cite[Proposition
3.8]{AET}:

\begin{lemma}
\label{lem_conversion}Assume that $\mu$ is a trace. Let $a,b\in A.$ Then
\[
\bar{\mu}(b^{\ast}Pb\bar{\alpha}^{n}(aPa^{\ast}))=\lambda(|{D(b\alpha
^{n}(a))|}^{2}).
\]

\begin{proof}
$\bar{\mu}(b^{\ast}Pb\bar{\alpha}^{n}(aPa^{\ast}))=\bar{\mu}(D(c)PD(c^{\ast
}))=\mu(D(c)D(c^{\ast}))$ in terms of $c:=b\alpha^{n}(a)$.
\end{proof}
\end{lemma}

The following is a version of the van der Corput lemma:

\begin{lemma}
{\cite[Lemma 2.12.7]{Tao}}\label{lem_vdC} Let $(v_{n})$ be a bounded sequence
of vectors in a Hilbert space $\mathfrak{H}$ such that
\begin{equation}
\lim_{M\rightarrow\infty}\frac{1}{M}\sum_{h=1}^{M}\left(  \limsup
_{N\rightarrow\infty}\frac{1}{N}\sum_{n=1}^{N}\langle v_{n},v_{n+h}%
\rangle\right)  =0. \label{eq_vdC_suff}%
\end{equation}
Then
\[
\lim_{N\rightarrow\infty}\frac{1}{N}\sum_{n=1}^{N}v_{n}=0.
\]
\end{lemma}

Putting these two lemmas together, we obtain the following:

\begin{lemma}
\label{lem_one_implies_two_mixing} Assume $\mu$ is a trace. Let $a\in A$
satisfy
\begin{equation}
\lim_{N\rightarrow\infty}\frac{1}{N}\sum_{n=1}^{N}\lambda(|{D(a}^{\ast}%
{\alpha^{n}(a))|}^{2})=0. \label{eq_one_weak_mixing}%
\end{equation}
Then, for all $b\in A,$ we have
\begin{equation}
\lim_{N\rightarrow\infty}\frac{1}{N}\sum_{n=1}^{N}\lambda(|{D(b\alpha
^{n}(a))|}^{2})=0. \label{eq_two_weak_mixing}%
\end{equation}

\begin{proof}
Let $x:=aPa^{\ast}$ and $y:=b^{\ast}Pb$. Observe that $\bar{\mu}(y\bar{\alpha
}^{n}(x))\geq0$ by Lemma \ref{lem_conversion}, and
\begin{align}
\frac{1}{N}\sum_{n=1}^{N}\bar{\mu}(y\bar{\alpha}^{n}(x))^{2}  &  =\frac{1}
{N}\sum_{n=1}^{N}\langle\hat{y},\bar{U}^{n}\hat{x}\rangle^{2}=\left\langle
\hat{y},\frac{1}{N}\sum_{n=1}^{N}{\langle\hat{y},\bar{U}^{n}\hat{x}\rangle
}\bar{U}^{n}\hat{x}\right\rangle \nonumber\\
&  \leq{\left\|  \bar{y}\right\|  }{\left\|  \frac{1}{N}\sum_{n=1}^{N}
{\langle\hat{y},\bar{U}^{n}\hat{x}\rangle}\bar{U}^{n}\hat{x}\right\|  }
\label{eq_prelim_calc}
\end{align}

Let $v_{n}:={\langle\hat{y},\bar{U}^{n}\hat{x}\rangle}\bar{U}^{n}\hat{x}$, 
for every $n\in\mathbb{N}$. Clearly, the sequence $(v_{n})$ is bounded.
We can estimate, for every $n,h\in\mathbb{N}$,
\[
|\left\langle v_{n},v_{n+h}\right\rangle |\leq\left\|  \hat{x}\right\|
^{2}\left\|  \hat{y}\right\|  ^{2}\bar{\mu}(x\bar{\alpha}^{h}(x)).
\]
This, together with Lemma \ref{lem_conversion} (with $b=a^{\ast}$) and our
assumption Eq. (\ref{eq_one_weak_mixing}), imply Eq. (\ref{eq_vdC_suff}).
Thus, from Lemma \ref{lem_vdC}, we have $\lim_{N\rightarrow\infty}\frac{1}%
{N}\sum_{n=1}^{N}v_{n}=0$. Therefore, from (\ref{eq_prelim_calc}), we obtain
\[
\lim_{N\rightarrow\infty}\frac{1}{N}\sum_{n=1}^{N}\bar{\mu}(y\bar{\alpha}%
^{n}(x))^{2}=0.
\]
Consequently, from (\ref{kwad}),
\[
\lim_{N\rightarrow\infty}\sum_{n=1}^{N}\bar{\mu}(y\bar{\alpha}^{n}(x))=0.
\]
Again by Lemma \ref{lem_conversion}, we are done.
\end{proof}
\end{lemma}

This finally implies the following characterization of relative weak mixing
(which in \cite{AET} was used as the definition):

\begin{proposition}
\label{prop_equiv_one_implies_two_mixing} Assume that $\mu$ is a trace. Then
$\mathbf{A}$ is weakly mixing relative to the subsystem $\mathbf{F}$ if and
only if
\[
\lim_{N\rightarrow\infty}\frac{1}{N}\sum_{n=1}^{N}\lambda(|{D(a}^{\ast}
{\alpha^{n}(a))|}^{2})=0,
\]
for all $a\in A$ such that $D(a)=0$.
\end{proposition}

\begin{remark}
Essential to the proof of the commutative version of Lemma
\ref{lem_one_implies_two_mixing} (outlined in \cite[Exercise 2.14.1]{Tao}), is
a conditional version of the Cauchy-Schwarz inequality in terms of the
conditional expectation $\mathbb{E}$:
\begin{equation*}
|\mathbb{E}(\bar{f}g|Y)|\leq{\lVert}\mathbb{E}{(|f|^{2}|Y)\rVert}_{L^{2}
(X|Y)}{\lVert}\mathbb{E}{(|g|^{2}|Y)\rVert}_{L^{2}(X|Y)}
\label{eq_cond_cauchy}
\end{equation*}
where $f,g$ belong to the $L^{\infty}(Y)$-module $L^{2}(X|Y)=\{h\in
L^{2}(X):\mathbb{E}(|h|^{2}|Y)\in L^{\infty}(Y)\}$ (\cite[Section 2.13]{Tao}).
In the non-commutative case, however, our approach above allows us to simplify 
the argument and avoid some snags. We essentially used a non-commutative 
translation of the proof of the absolute case \cite[Corollary 2.12.8]{Tao}, 
but in terms of the basic construction, to prove 
Lemma \ref{lem_one_implies_two_mixing}.
\end{remark}

Before we get to an example, we note a few simple general facts:

Firstly, $D(a)=0$ for $a\in A$, if and only if $a$ is of the form $a=c-D(c)$
for some $c\in A$.

Secondly,
\[
\lambda(D(\alpha^{n}(a^{\ast})b^{\ast})D\left(  b\alpha^{n}(a))\right)
=\left\|  PbU^{n}a\Omega\right\|  ^{2}%
\]
for all $a,b\in A$, by a straightforward calculation. If, in addition
$\lambda$ is a trace, then we have%
\begin{equation}
\left\|  PbU^{n}a\Omega\right\|  =\left\|  PU^{n}a^{\ast}U^{-n}b^{\ast}%
\Omega\right\|  \label{normVirSp}%
\end{equation}
for all $a,b\in A$, by a similar calculation for $\lambda(D\left(  b\alpha
^{n}(a))D(\alpha^{n}(a^{\ast})b^{\ast})\right)  $.

To show that relative weak mixing is indeed relevant in noncommutative
W*-dynamical systems, in particular for non-ergodic systems which are not
asymptotically abelian, we provide the following example:

\begin{example}
Let $G$ be any discrete group, and let $A$ be the group von Neumann algebra
obtained from it. In other words, $A$ is the von Neumann algebra on
$H=l^{2}(G)$ generated by the following set of unitary operators:%
\[
\{l(g):g\in G\}
\]
where $l$ is the left regular representation of $G$, i.e. the unitary
representation of $G$ on $H$ with each $l(g):H\rightarrow H$ given by
\[
\lbrack l(g)f](h)=f(g^{-1}h)
\]
for all $f\in H$ and $g,h\in G$. Equivalently,%
\[
l(g)\delta_{h}=\delta_{gh}%
\]
for all $g,h\in G$, where $\delta_{g}\in H$ is defined by $\delta_{g}(g)=1$
and $\delta_{g}(h)=0$ for $h\neq g$. Setting
\[
\Omega:=\delta_{1}%
\]
where $1\in G$ denotes the identity of $G$, we can define a faithful normal
trace $\mu$ on $A$ by%
\[
\mu(a):=\left\langle \Omega,a\Omega\right\rangle
\]
for all $a\in A$. Then $(H,\operatorname{id}_{A},\Omega)$ is the cyclic
representation of $(A,\mu)$.

Given any automorphism $T$ of $G$, we define a unitary operator on $H$ by
\[
Uf:=f\circ T^{-1}%
\]
for all $f\in H$. From this we obtain a $\ast$-automorphism of $A$ by setting
\[
\alpha(a):=UaU^{\ast}%
\]
for all $a\in A$, which satisfies $\alpha(l(g))=l(T(g))$ for all $g\in G$.
Then $\mathbf{A}=(A,\mu,\alpha)$ is a system.

We define a subsystem $\mathbf{F}=(F,\lambda,\varphi)$ of $\mathbf{A}$ by
letting $F$ be the von Neumann subalgebra of $A$ generated by%
\[
\{l(g):g\in K\}
\]
where $K:=\{g\in G:T^{\mathbb{N}}(g)$ is finite$\}$. Here $T^{\mathbb{N}%
}(g):=\{T(g),T^{2}(g),T^{3}(g),...\}$ is the \emph{orbit} of $g$. Furthermore
$\lambda:=\mu|_{F}$ and $\varphi:=\alpha|_{A}$. (See \cite[Section 3]{D2} for
more background on this type of system in the context of W*-algebraic ergodic
theory and joinings.)

We can find $D$ explicitly in this case: The projection $P$ above is now the
projection of $H$ onto the Hilbert subspace spanned by $\{\delta_{g}:g\in
K\}$. Therefore we have
\begin{equation}
D(l(g))=\left\{
\begin{array}
[c]{c}%
l(g)\text{ for }g\in K\\
0\text{ for }g\notin K
\end{array}
\right.  \text{ } \label{Dformule}
\end{equation}
for all $g\in G$.

Note that the unital $\ast$-algebra generated by $\{l(g):g\in G\}$ is exactly
$A_{0}=\operatorname{span}\{l(g):g\in G\}$.

Suppose that for any $g,h\in G$ with $g\notin K$, it is true that%
\begin{equation}
D(l(hT^{n}(g)))=0 \label{D=0}%
\end{equation}
for $n$ large enough, i.e. for $n>n_{0}$ for some $n_{0}$. Then, for any
$c_{0},b_{0}\in A_{0}$, and $a_{0}:=c_{0}-D(c_{0})$, we have
\[
Pb_{0}U^{n}a_{0}\Omega=0
\]
for $n$ large enough. Since $A_{0}$ is strongly dense in $A$, it follows that%
\[
\lim_{n\rightarrow\infty}Pb_{0}U^{n}a\Omega=0
\]
for all $a\in A$ such that $D(a)=0$, by simply considering any $c\in A$ and
some $c_{0}\in A_{0}$ such that $\left\|  c_{0}\Omega-c\Omega\right\|
<\varepsilon$ for an $\varepsilon>0$ of our choosing, and setting $a:=c-D(c)$.

Since $\lambda$ is a trace, we can apply a similar argument to $\left\|
Pb_{0}U^{n}a\Omega\right\|  =\left\|  PU^{n}a^{\ast}U^{-n}b_{0}^{\ast}%
\Omega\right\|  $ (see Eq. (\ref{normVirSp})) to show that
\[
\lim_{n\rightarrow\infty}PbU^{n}a\Omega=0
\]
and therefore
\[
\lim_{n\rightarrow\infty}\lambda(D(\alpha^{n}(a^{\ast})b^{\ast})D\left(
b\alpha^{n}(a))\right)  =0
\]
for all $a,b\in A$ such that $D(a)=0$. It follows easily from this that
$\mathbf{A}$ is weakly mixing relative to $\mathbf{F}$. The limit above could
be interpreted as $\mathbf{A}$ having a stronger property, namely that
$\mathbf{A}$ is ``strongly mixing relative to $\mathbf{F}$''.

What remains is to show specific cases for which Eq. (\ref{D=0}) holds and
which illustrate the points made above about noncommutative systems.

A simple case is when $G$ is the free group on a countably infinite set of
symbols $S$. We then consider any bijection $T:S\rightarrow S$ which has both
finite and infinite orbits in $S$, say $T$ is a permutation when restricted to
some finite non-empty subset, or to each of infinitely many finite non-empty 
subsets, while it shifts the remaining infinite subset of $S$. We obtain a 
automorphism $T$ of $G$ from this bijection. Then Eq. (\ref{D=0}) follows from 
Eq. (\ref{Dformule}).

But at the same time, $F$ is then not trivial, i.e. $F$ strictly contains the
subalgebra $\mathbb{C}1$, and is in general not abelian. In fact, $F$ is
$\ast$-isomorphic to the group von Neumann algebra of the free group $K$ on the 
symbols with finite orbits. That $F\neq\mathbb{C}1$, also implies that 
$\mathbf{A}$ is not ergodic (see \cite[Theorem 3.4]{D2}). Furthermore,
\[
\left\|  \lbrack\alpha^{n}(l(g)),l(h)]\Omega\right\|  =\sqrt{2}%
\]
if $T^{n}(g)h\neq hT^{n}(g)$, which is the case if $g$ and $h$ are in two
separate orbits, or if $g=h$ has an infinite orbit. Hence $\mathbf{A}$ is not
asymptotically abelian in the sense of \cite[Definition 1.10]{AET}.
Furthermore, $A$ is a factor.
\end{example}

\section{The joining characterization\label{afdBewys}}

This section presents the main result of the paper, still using the notation
from Section \ref{afdROB}.

Let $H_{\omega}^{W}$ denote the fixed point space of $W$. The relative
independent joining (or the relative product system) will connect to relative
weak mixing via the following notion:

\begin{definition}
We say that $\mathbf{A}\odot_{\mathbf{F}}\mathbf{A}^{\prime}$ \emph{is ergodic
relative to} the modular subsystem $\mathbf{F}$ of $\mathbf{A}$, if
$H_{\omega}^{W}\subset H_{\lambda}$.
\end{definition}

Our main goal in this paper is to prove the following characterization of
relative weak mixing:

\begin{theorem}
\label{st173}Assume that $\mu$ is a trace. Then $\mathbf{A}$ is weakly mixing
relative to $\mathbf{F}$ if and only if $\mathbf{A}\odot_{\mathbf{F}%
}\mathbf{A}^{\prime}$ is ergodic relative to $\mathbf{F}$.
\end{theorem}

The rest of this section is devoted to the proof of this theorem. We break the
proof into a sequence of smaller results. Some of these are of independent
interest (in particular Propositions \ref{prop157}, \ref{prop172} and
\ref{prop161}, and Remark \ref{relErgKar}), and do not require $\mu$ to
be tracial.

The following lemma and proposition proves one direction of Theorem
\ref{st173}. In the classical case, this direction is also proven in
\cite[Proposition 6.2]{F81} and \cite[Lemma 1.3]{FK}, but using different 
arguments.

\begin{lemma}
\label{lem152}Consider a modular subsystem $\mathbf{F}$ of the system
$\mathbf{A}$. For any $a\in A$ with $D(a)=0$ and any $b\in A^{\prime}$, we
have%
\[
\pi_{\omega}(a\otimes b)\Omega_{\omega}\perp H_{\lambda}.
\]
\end{lemma}

\begin{proof}
For any $c\in F$,%
\begin{align*}
\left\langle \pi_{\lambda}(c)\Omega_{\omega},\pi_{\mu}(a)\Omega_{\omega
}\right\rangle  &  =\left\langle \Omega_{\omega},\pi_{\mu}(c^{\ast}%
a)\Omega_{\omega}\right\rangle =\mu(c^{\ast}a)\\
&  =\lambda(D(c^{\ast}a))=\lambda(c^{\ast}D(a))\\
&  =0.
\end{align*}
Hence, $\pi_{\mu}(a)\Omega_{\omega}\in H_{\mu}\ominus H_{\lambda}$. So
$\pi_{\mu}(a)\Omega_{\omega}\perp H_{\mu^{\prime}}$ by \cite[Proposition
3.6]{D3}. On the other hand, $\pi_{\mu^{\prime}}(b^{\ast}f)\Omega_{\omega}\in
H_{\mu^{\prime}}$ for any $f\in\tilde{F}$, so $\left\langle \pi_{\mu^{\prime}%
}(b^{\ast}f)\Omega_{\omega},\pi_{\mu}(a)\Omega_{\omega}\right\rangle =0$.
Therefore,
\begin{align*}
\left\langle \pi_{\tilde{\lambda}}(f)\Omega_{\omega},\pi_{\omega}(a\otimes
b)\Omega_{\omega}\right\rangle  &  =\left\langle \pi_{\omega}(1\otimes
b^{\ast})\pi_{\mu^{\prime}}(f)\Omega_{\omega},\pi_{\omega}(a\otimes
1)\Omega_{\omega}\right\rangle \\
&  =\left\langle \pi_{\mu^{\prime}}(b^{\ast}f)\Omega_{\omega},\pi_{\mu
}(a)\Omega_{\omega}\right\rangle \\
&  =0,
\end{align*}
proving the lemma, since $\pi_{\tilde{\lambda}}(\tilde{F})\Omega_{\omega}$ is 
dense in $H_\lambda$.
\end{proof}

Using this lemma we can show one direction of Theorem \ref{st173}:

\begin{proposition}
\label{prop151}Assume that $\mu$ is a trace and that $\mathbf{A}%
\odot_{\mathbf{F}}\mathbf{A}^{\prime}$ is ergodic relative to $\mathbf{F}$.
Then
\[
\lim_{N\rightarrow\infty}\frac{1}{N}\sum_{n=1}^{N}\lambda\left(  |D\left(
b\alpha^{n}(a)\right)  |^{2}\right)  =0
\]
for all $a,b\in A$ such that $D(a)=0$ or $D(b)=0$.
\end{proposition}

\begin{proof}
Let $Q$ be the projection of $H_{\omega}$ onto the fixed point space
$H_{\omega}^{W}$ of $W$. By the mean ergodic theorem we then have
\begin{align*}
\lim_{N\rightarrow\infty}\frac{1}{N}\sum_{n=1}^{N}\omega(\tau^{n}(s)t)  &
=\lim_{N\rightarrow\infty}\frac{1}{N}\sum_{n=1}^{N}\left\langle W^{n}%
\pi_{\omega}(s^{\ast})\Omega_{\omega},\pi_{\omega}(t)\Omega_{\omega
}\right\rangle \\
&  =\left\langle Q\pi_{\omega}(s^{\ast})\Omega_{\omega},\pi_{\omega}%
(t)\Omega_{\omega}\right\rangle
\end{align*}
for all $s,t\in A\odot A^{\prime}$. This holds in particular for $s=a^{\ast
}\otimes j(a)$ and $t=b^{\ast}\otimes j(b)$, where $a,b\in A$, and $D(a)=0$ or
$D(b)=0$.

Suppose $D(a)=0$ (the case $D(b)=0$ is similar, by taking $Q$ to the other
side in the inner product above). Then $\pi_{\omega}(s^{\ast})\Omega_{\omega
}\perp H_{\omega}^{W}$ by Lemma \ref{lem152}, so $Q\pi_{\omega}(s^{\ast
})\Omega_{\omega}=0$. This means, by the definition of $\omega=\mu
\odot_{\lambda}\mu^{\prime}$ in Eq. (\ref{mumu}), that%
\begin{align*}
0  &  =\lim_{N\rightarrow\infty}\frac{1}{N}\sum_{n=1}^{N}\left\langle
\Omega,D(\alpha^{n}(a^{\ast})b^{\ast})\tilde{D}(\alpha^{\prime n}%
(j(a))j(b))\Omega\right\rangle \\
&  =\lim_{N\rightarrow\infty}\frac{1}{N}\sum_{n=1}^{N}\lambda(D(\alpha
^{n}(a^{\ast})b^{\ast})D(b\alpha^{n}(a))),
\end{align*}
as required, since%
\[
\tilde{D}(\alpha^{\prime n}(j(a))j(b)\Omega=PJ\alpha^{n}(a^{\ast})b^{\ast
}\Omega=D(b\alpha^{n}(a))\Omega,
\]
where we have used the fact that $\mu$ is a trace (so $Jc\Omega=c^{\ast}%
\Omega$ for all $c\in A$).
\end{proof}

Next we consider the other direction of Theorem \ref{st173}. We don't have a
reference to a proof of the classical case of this direction. Our first step
is the following:

\begin{proposition}
\label{prop157}$\mathbf{A}\odot_{\mathbf{F}}\mathbf{A}^{\prime}$ is ergodic
relative to $\mathbf{F}$ if and only if
\begin{equation}
\lim_{N\rightarrow\infty}\frac{1}{N}\sum_{n=1}^{N}\omega(t\tau^{n}%
(s))=\lim_{N\rightarrow\infty}\frac{1}{N}\sum_{n=1}^{N}\omega(E(t)\tau
^{n}(E(s))) \label{tensRSV}%
\end{equation}
for all $s,t\in A\odot A^{\prime}$. Both limits exist, whether $\mathbf{A}%
\odot_{\mathbf{F}}\mathbf{A}^{\prime}$ is ergodic relative to $\mathbf{F}$ or not.
\end{proposition}

\begin{proof}
Let $Q$ be the projection of $H_{\omega}$ onto the fixed point space
$H_{\omega}^{W}$ of $W$. Let $R$ be the projection of $H_{\omega}$ onto
$H_{\lambda}$.

By the mean ergodic theorem, for all $s,t\in A\odot A^{\prime}$,%
\[
\lim_{N\rightarrow\infty}\frac{1}{N}\sum_{n=1}^{N}\omega(t\tau^{n}%
(s))=\left\langle \gamma_{\omega}(t^{\ast}),Q\gamma_{\omega}(s)\right\rangle
\]

and%
\[
\lim_{N\rightarrow\infty}\frac{1}{N}\sum_{n=1}^{N}\omega(E(t)\tau
^{n}(E(s)))=\left\langle \gamma_{\omega}(E(t^{\ast})),Q\gamma_{\omega
}(E(s))\right\rangle .
\]

Let $P_{\mu}$ be the projection of $H_{\mu}\ $onto $H_{\lambda}$, and
$P_{\mu^{\prime}}$ the projection of $H_{\mu^{\prime}}$ onto $H_{\lambda}$.
Consider $s=a\otimes b$, where $a\in A$ and $b\in A^{\prime}$. Then, because
$\pi_{\mu^{\prime}}(\tilde{D}(b))\Omega_{\omega}\in H_{\lambda}$, we know by
the construction of $D$ (see for example \cite[Section 10.2]{S}) that
\begin{align*}
\gamma_{\omega}(E(s)) &  =\pi_{\mu}(D(a))\pi_{\mu^{\prime}}(\tilde
{D}(b))\Omega_{\omega}=\pi_{\lambda}(D(a))\pi_{\mu^{\prime}}(\tilde
{D}(b))\Omega_{\omega}\\
&  =P_{\mu}\pi_{\mu}(a)\pi_{\mu^{\prime}}(\tilde{D}(b))\Omega_{\omega}=P_{\mu
}\pi_{\mu}(a)P_{\mu^{\prime}}\pi_{\mu^{\prime}}(b)\Omega_{\omega}\\
&  =R\pi_{\mu}(a)R\pi_{\mu^{\prime}}(b)\Omega_{\omega}.
\end{align*}
since $R|_{H_{\mu}}=P_{\mu}$ and $R|_{H_{\mu^{\prime}}}=P_{\mu^{\prime}}$.

For $y\in H_{\mu^{\prime}}\ominus H_{\lambda}$ and $f\in F$, we have
\[
\left\langle \pi_{\lambda}(f)\Omega_{\omega},\pi_{\omega}(a\otimes
1)y\right\rangle =\left\langle \pi_{\mu}(a^{\ast}f)\Omega_{\omega
},y\right\rangle =0,
\]
since $\pi_{\mu}(a^{\ast}f)\Omega_{\omega}\in H_{\mu}\perp(H_{\mu^{\prime}%
}\ominus H_{\lambda})$ by \cite[Proposition 3.6]{D3}. So $\pi_{\omega
}(a\otimes1)y\perp H_{\lambda}$, which means that
\[
\gamma_{\omega}(E(s))=R\pi_{\omega}(a\otimes1)\pi_{\mu^{\prime}}%
(b)\Omega_{\omega}=R\pi_{\omega}(a\otimes b)\Omega_{\omega}.
\]
So
\[
\gamma_{\omega}(E(s))=R\gamma_{\omega}(s)
\]
for all $s\in A\odot A^{\prime}$. Hence,
\[
\left\langle \gamma_{\omega}(E(t^{\ast})),Q\gamma_{\omega}(E(s))\right\rangle
=\left\langle R\gamma_{\omega}(t^{\ast}),QR\gamma_{\omega}(s)\right\rangle
\]
for all $s,t\in A\odot A^{\prime}$.

Now, if $\mathbf{A}\odot_{\mathbf{F}}\mathbf{A}^{\prime}$ is ergodic relative
to $\mathbf{F}$, i.e. $Q\leq R$, it follows that%
\[
\left\langle \gamma_{\omega}(E(t^{\ast})),Q\gamma_{\omega}(E(s))\right\rangle
=\left\langle \gamma_{\omega}(t^{\ast}),Q\gamma_{\omega}(s)\right\rangle
\]
from which we see that Eq. (\ref{tensRSV}) holds for all $s,t\in A\odot
A^{\prime}$.

Conversely, if Eq. (\ref{tensRSV}) holds for all $s,t\in A\odot A^{\prime}$,
then we have%
\[
\left\langle R\gamma_{\omega}(t^{\ast}),QR\gamma_{\omega}(s)\right\rangle
=\left\langle \gamma_{\omega}(t^{\ast}),Q\gamma_{\omega}(s)\right\rangle
\]
for all $s,t\in A\odot A^{\prime}$. It follows that $RQR=Q$, so $Q\leq R$,
meaning $\mathbf{A}\odot_{\mathbf{F}}\mathbf{A}^{\prime}$ is ergodic relative
to $\mathbf{F}$.
\end{proof}

As a consequence of this proposition, we have the following lemma towards the
proof of Theorem \ref{st173}:

\begin{lemma}
\label{prop160}Assume that $\mu$ is a trace. Then $\mathbf{A}\odot
_{\mathbf{F}}\mathbf{A}^{\prime}$ is ergodic relative to $\mathbf{F}$ if and
only if
\begin{equation}
\lim_{N\rightarrow\infty}\frac{1}{N}\sum_{n=1}^{N}\lambda(|D(b\alpha
^{n}(a))|^{2})=\lim_{N\rightarrow\infty}\frac{1}{N}\sum_{n=1}^{N}%
\lambda(|D(b)D(\alpha^{n}(a))|^{2}) \label{prodRSV}%
\end{equation}
for all $a,b\in A$. Both limits exist, whether $\mathbf{A}\odot_{\mathbf{F}%
}\mathbf{A}^{\prime}$ is ergodic relative to $\mathbf{F}$ or not.
\end{lemma}

\begin{proof}
Suppose $\mathbf{A}\odot_{\mathbf{F}}\mathbf{A}^{\prime}$ is ergodic relative
to $\mathbf{F}$, then Eq. (\ref{tensRSV}) holds. Applying it to $s=a\otimes c$
and $t=b\otimes d$, for $a,b\in A$ and $c,d\in A^{\prime}$, we obtain%
\begin{align*}
&  \lim_{N\rightarrow\infty}\frac{1}{N}\sum_{n=1}^{N}\omega([b\alpha
^{n}(a)]\otimes\lbrack d\alpha^{\prime n}(c)])\\
&  =\lim_{N\rightarrow\infty}\frac{1}{N}\sum_{n=1}^{N}\omega([D(b)D(\alpha
^{n}(a))]\otimes\lbrack\tilde{D}(d)\tilde{D}(\alpha^{\prime n}(c))]).
\end{align*}
Using the definition of $\omega$, this is equivalent to%
\begin{align*}
&  \lim_{N\rightarrow\infty}\frac{1}{N}\sum_{n=1}^{N}\left\langle
\Omega,D(b\alpha^{n}(a))\tilde{D}(d\alpha^{\prime n}(c))\Omega\right\rangle \\
&  =\lim_{N\rightarrow\infty}\frac{1}{N}\sum_{n=1}^{N}\left\langle
\Omega,D(b)D(\alpha^{n}(a))\tilde{D}(d)\tilde{D}(\alpha^{\prime n}%
(c))\Omega\right\rangle .
\end{align*}
Setting $c=j(a^{\ast})=JaJ$ and $d=j(b^{\ast})=JbJ$, we have in particular%
\begin{align*}
&  \lim_{N\rightarrow\infty}\frac{1}{N}\sum_{n=1}^{N}\left\langle
\Omega,D(b\alpha^{n}(a))JD(b\alpha^{n}(a))\Omega\right\rangle \\
&  =\lim_{N\rightarrow\infty}\frac{1}{N}\sum_{n=1}^{N}\left\langle
\Omega,D(b)D(\alpha^{n}(a))JD(b)D(\alpha^{n}(a))\Omega\right\rangle .
\end{align*}
Since $\mu$ is a trace, this is equivalent to
\begin{align*}
&  \lim_{N\rightarrow\infty}\frac{1}{N}\sum_{n=1}^{N}\left\langle
\Omega,D(b\alpha^{n}(a))D(\alpha^{n}(a^{\ast})b^{\ast})\Omega\right\rangle \\
&  =\lim_{N\rightarrow\infty}\frac{1}{N}\sum_{n=1}^{N}\left\langle
\Omega,D(b)D(\alpha^{n}(a))D(\alpha^{n}(a^{\ast}))D(b^{\ast})\Omega
\right\rangle .
\end{align*}
Since $\lambda$ is a trace, this is equivalent to Eq. (\ref{prodRSV}).

Note that from the manipulations above we also see that
\[
\lim_{N\rightarrow\infty}\frac{1}{N}\sum_{n=1}^{N}\lambda(|D(b\alpha
^{n}(a))|^{2})=\lim_{N\rightarrow\infty}\frac{1}{N}\sum_{n=1}^{N}\omega
(t\tau^{n}(s))
\]
and%
\[
\lim_{N\rightarrow\infty}\frac{1}{N}\sum_{n=1}^{N}\lambda(|D(b)D(\alpha
^{n}(a))|^{2})=\lim_{N\rightarrow\infty}\frac{1}{N}\sum_{n=1}^{N}%
\omega(E(t)\tau^{n}(E(s))
\]
exist by Proposition \ref{prop157}, whether $\mathbf{A}\odot_{\mathbf{F}%
}\mathbf{A}^{\prime}$ is ergodic relative to $\mathbf{F}$ or not, where
$s=a\otimes(JaJ)$ and $t=b\otimes(JbJ)$.

Now, suppose Eq. (\ref{prodRSV}) holds, then we have by the equivalences
above, that%
\begin{align*}
&  \lim_{N\rightarrow\infty}\frac{1}{N}\sum_{n=1}^{N}\omega([b\alpha
^{n}(a)]\otimes\lbrack JbJ\alpha^{\prime n}(JaJ)])\\
&  =\lim_{N\rightarrow\infty}\frac{1}{N}\sum_{n=1}^{N}\omega([D(b)D(\alpha
^{n}(a))]\otimes\lbrack\tilde{D}(JbJ)\tilde{D}(\alpha^{\prime n}(JaJ))]),
\end{align*}
i.e.
\begin{align*}
&  \lim_{N\rightarrow\infty}\frac{1}{N}\sum_{n=1}^{N}\omega((b\otimes
(JbJ))\tau^{n}(a\otimes(JaJ)))\\
&  =\lim_{N\rightarrow\infty}\frac{1}{N}\sum_{n=1}^{N}\omega(E(b\otimes
(JbJ))\tau^{n}(E(a\otimes(JaJ)))).
\end{align*}
Because of the polarization identity, applied in turn to the two appearances
of the sesquilinear form $A\times A\ni(a,c)\mapsto a\otimes(JcJ)$ above (once
inside $\tau^{n}$ and once outside), Eq. (\ref{tensRSV}) then follows, so
$\mathbf{A}\odot_{\mathbf{F}}\mathbf{A}^{\prime}$ is ergodic relative to
$\mathbf{F}$ by Proposition \ref{prop157}.
\end{proof}

In order to proceed, we need the notion of relative ergodicity for a system itself:

\begin{definition}
\label{relErg}We say that $\mathbf{A}$ \emph{is ergodic relative to
}$\mathbf{F}$ if $H^{U}\subset H_{F}$, where $H^{U}$ is the fixed point space
of $U:H\rightarrow H$, and $H_{F}=\overline{F\Omega}$.
\end{definition}

This generalizes ergodicity of $\mathbf{A}$, which is the special case
$H^{U}=\mathbb{C}\Omega$.

\begin{remark}
\label{relErgKar}In \cite[Definition 4.1]{D3} an alternative condition was
used instead of $H^{U}\subset H_{F}$ to define relative ergodicity, namely
\[
A^{\alpha}\subset F,
\]
where $A^{\alpha}:=\{a\in A:\alpha(a)=a\}$. For our purposes here, Definition
\ref{relErg} is the more convenient definition, but the question nevertheless
arises whether the two conditions are equivalent. From \cite[Proposition
4.2]{D3} we know that $H^{U}=\overline{A^{\alpha}\Omega}$, so if $A^{\alpha
}\subset F$, then $H^{U}\subset H_{F}$. This fact is used in Proposition
\ref{prop172}.

We do not need the converse. However, it does hold, since $\mathbf{F}$ is a
modular subsystem, as we now explain. The conditional expectation $D$ is
determined by%
\[
D(a)|_{H_{F}}=Pa|_{H_{F}}%
\]
for all $a\in A$; see for example \cite[Section 10.2]{S}. The subalgebra
$A^{\alpha}$ is easily seen to be globally invariant under the modular group
as well (see \cite[Proposition 4.2]{D3}), hence we also have a unique
conditional expectation $D_{A^{\alpha}}:A\rightarrow A^{\alpha}$ such that
$\mu\circ D_{A^{\alpha}}=\mu$, which is similarly determined by%
\[
D_{A^{\alpha}}(a)|_{H^{U}}=Qa|_{H^{U}}%
\]
where $Q$ is the projection of $H$ onto $H^{U}$. Assuming $H^{U}\subset H_{F}%
$, it follows that
\[
D(D_{A^{\alpha}}(a))|_{H^{U}}=PD_{A^{\alpha}}(a)|_{H^{U}}=Qa|_{H^{U}%
}=D_{A^{\alpha}}(a)|_{H^{U}}%
\]
and therefore $D(D_{A^{\alpha}}(a))=D_{A^{\alpha}}(a)$, since $\Omega\in
H^{U}$ is separating for $A$. So, for $a\in A^{\alpha}$, we have
\[
a=D_{A^{\alpha}}(a)=D(D_{A^{\alpha}}(a))\in F
\]
which means that $A^{\alpha}\subset F$.

To summarize: $\mathbf{A}$ is ergodic relative to $\mathbf{F}$, if and only if
$A^{\alpha}\subset F$.
\end{remark}

The following generalizes the standard fact that weak mixing implies ergodicity:

\begin{proposition}
\label{prop172}If $\mathbf{A}$ is weakly mixing relative to $\mathbf{F}$, then
$\mathbf{A}$ is ergodic relative to $\mathbf{F}$.
\end{proposition}

\begin{proof}
From Proposition \ref{prop140}, we have
$\lambda(|D(ba) -D(b)D(a)|^{2})  =0$ for $a\in A^{\alpha}$ and all $b\in A$. Since
$\lambda$ is faithful, it follows that $D(b(a-D(a)))=D\left(  ba\right)
-D(b)D(a)=0$. In particular, setting $b=(a-D(a))^{\ast}$, we conclude that
$a=D(a)\in F$, since $\mu$ is faithful and $\lambda\circ D=\mu$. So
$A^{\alpha}\subset F$, hence $H^{U}\subset H_{F}$ by the first part of Remark
\ref{relErgKar}.
\end{proof}

Next we consider a version of Proposition \ref{prop157} for a system itself.

\begin{proposition}
\label{prop161}$\mathbf{A}$ is ergodic relative to $\mathbf{F}$ if and only if%
\begin{equation}
\lim_{N\rightarrow\infty}\frac{1}{N}\sum_{n=1}^{N}\mu(b\alpha^{n}%
(a))=\lim_{N\rightarrow\infty}\frac{1}{N}\sum_{n=1}^{N}\lambda(D(b)\alpha
^{n}(D(a)))\label{relErgKar2}%
\end{equation}
for all $a,b\in A$. Both limits exist, whether $\mathbf{A}$ is ergodic
relative to $\mathbf{F}$ or not.
\end{proposition}

\begin{proof}
Essentially the same argument, using the mean ergodic theorem, as in the proof
of Proposition \ref{prop157}, but with $Q$ now the projection of $H$ onto
$H^{U}$, and with $R$ replaced by $P$.
\end{proof}

Using the last three results, we can now prove the remaining direction of
Theorem \ref{st173}:

\begin{proposition}
\label{prop162}Assume that $\mu$ is tracial and that $\mathbf{A}$ is weakly
mixing relative to $\mathbf{F}$. Then $\mathbf{A}\odot_{\mathbf{F}}%
\mathbf{A}^{\prime}$ is ergodic relative to $\mathbf{F}$.
\end{proposition}

\begin{proof}
Note that for all $a,b\in A$,%
\begin{align*}
\lambda\left(  |D\left(  b\alpha^{n}(a)\right)  -D(b)D(\alpha^{n}%
(a))|^{2}\right)   &  =\lambda(|D(b\alpha^{n}(a))|^{2})\\
&  -\lambda(D(\alpha^{n}(a^{\ast})b^{\ast})D(b)D(\alpha^{n}(a)))\\
&  -\lambda(D(\alpha^{n}(a^{\ast}))D(b^{\ast})D(b\alpha^{n}(a)))\\
&  +\lambda(D(\alpha^{n}(a^{\ast}))D(b^{\ast})D(b)D(\alpha^{n}(a))).
\end{align*}

Consider the second term and use the trace property of $\mu$:%
\begin{align*}
\lambda(D(\alpha^{n}(a^{\ast})b^{\ast})D(b)D(\alpha^{n}(a))) &  =\lambda
(D(\alpha^{n}(a^{\ast})b^{\ast}D(b)D(\alpha^{n}(a))))\\
&  =\mu(\alpha^{n}(a^{\ast})b^{\ast}D(b)D(\alpha^{n}(a)))\\
&  =\mu(b^{\ast}D(b)\alpha^{n}(D(a)a^{\ast})).
\end{align*}
Since $\mathbf{A}$ is ergodic relative to $\mathbf{F}$ by Proposition
\ref{prop172}, we now have by Proposition \ref{prop161} that%
\begin{align*}
&  \lim_{N\rightarrow\infty}\frac{1}{N}\sum_{n=1}^{N}\lambda(D(\alpha
^{n}(a^{\ast})b^{\ast})D(b)D(\alpha^{n}(a)))\\
&  =\lim_{N\rightarrow\infty}\frac{1}{N}\sum_{n=1}^{N}\lambda(D(b^{\ast
})D(b)\alpha^{n}(D(a)D(a^{\ast})))\\
&  =\lim_{N\rightarrow\infty}\frac{1}{N}\sum_{n=1}^{N}\lambda(|D(b)D(\alpha
^{n}(a))|^{2})
\end{align*}
Similarly
\begin{align*}
&  \lim_{N\rightarrow\infty}\frac{1}{N}\sum_{n=1}^{N}\lambda(D(\alpha
^{n}(a^{\ast}))D(b^{\ast})D(b\alpha^{n}(a)))\\
&  =\lim_{N\rightarrow\infty}\frac{1}{N}\sum_{n=1}^{N}\lambda(|D(b)D(\alpha
^{n}(a))|^{2})
\end{align*}
and%
\begin{align*}
&  \lim_{N\rightarrow\infty}\frac{1}{N}\sum_{n=1}^{N}\lambda(D(\alpha
^{n}(a^{\ast}))D(b^{\ast})D(b)D(\alpha^{n}(a)))\\
&  =\lim_{N\rightarrow\infty}\frac{1}{N}\sum_{n=1}^{N}\mu(D(b^{\ast
})D(b)\alpha^{n}(aD(a^{\ast})))\\
&  =\lim_{N\rightarrow\infty}\frac{1}{N}\sum_{n=1}^{N}\lambda(|D(b)D(\alpha
^{n}(a))|^{2})
\end{align*}
Keep in mind that all these limits exist by Proposition \ref{prop161}. Then by
Proposition \ref{prop140},%
\begin{align*}
0 &  =\lim_{N\rightarrow\infty}\frac{1}{N}\sum_{n=1}^{N}\lambda\left(
|D\left(  b\alpha^{n}(a)\right)  -D(b)D(\alpha^{n}(a))|^{2}\right)  \\
&  =\lim_{N\rightarrow\infty}\frac{1}{N}\sum_{n=1}^{N}[\lambda(|D(b\alpha
^{n}(a))|^{2})-\lambda(|D(b)D(\alpha^{n}(a))|^{2})],
\end{align*}
so%
\[
\lim_{N\rightarrow\infty}\frac{1}{N}\sum_{n=1}^{N}\lambda(|D(b\alpha
^{n}(a))|^{2})=\lim_{N\rightarrow\infty}\frac{1}{N}\sum_{n=1}^{N}%
\lambda(|D(b)D(\alpha^{n}(a))|^{2}),
\]
since both limits exist (see Lemma \ref{prop160}). By Lemma \ref{prop160} we
are done.
\end{proof}

This completes the proof of Theorem \ref{st173}. To summarize: the one
direction is given by Proposition \ref{prop151}, the other by Proposition
\ref{prop162}.

To connect this to the structure theorem in \cite{AET}, we mention the
following: Suppose that we have an asymptotically abelian W*-dynamical system
$\mathbf{A}$ with a tracial invariant state, as defined in \cite[Definition
1.10]{AET}. According to \cite[Theorem 1.14]{AET} (and Proposition
\ref{prop_equiv_one_implies_two_mixing}), such a system is weakly mixing
relative to the \emph{central system} $\mathbf{C}:=(A\cap A^{\prime}%
,\mu|_{A\cap A^{\prime}},\alpha|_{A\cap A^{\prime}})$. Theorem \ref{st173}
then shows that $\mathbf{A}\odot_{\mathbf{C}}\mathbf{A}^{\prime}$ is ergodic
relative to $\mathbf{C}$.

\section{Acknowledgements}

This work was partially supported by the National Research Foundation of South Africa.

\end{document}